\documentclass[11pt, reqno]{amsart}
\usepackage{graphicx} 
\usepackage{amsmath}
\usepackage{amsthm}
\usepackage{amssymb}
\usepackage{amscd}
\usepackage{amsbsy}
\usepackage{epsfig}
\usepackage{graphicx}
\usepackage{psfrag}
\usepackage{bbm}
\usepackage{hyperref}
\usepackage{cleveref}
\usepackage[UKenglish]{babel}
\usepackage[UKenglish]{isodate}
\usepackage{enumerate}

\theoremstyle{plain}
\newtheorem{theorem}{Theorem}[section]
\newtheorem{proposition}[theorem]{Proposition}
\newtheorem{lemma}[theorem]{Lemma}

\newtheorem*{claim}{Claim}

\newtheorem*{lemma*}{Lemma}

\theoremstyle{plain}
\newtheorem{remark}[theorem]{Remark}

\theoremstyle{definition}
\newtheorem{definition}[theorem]{Definition}

\theoremstyle{definition}

\newtheorem{example}{Example}[section]

\def\R{\ensuremath{\mathbb R}}
\def\N{\ensuremath{\mathbb N}}
\def\Q{\ensuremath{\mathcal Q}}

\def\I{\ensuremath{{\bf 1}}}

\def\P{\ensuremath{\mathcal P}}

\def\supp{\ensuremath{\text{supp}}}

\def\X{\mathcal{X}}
\def\ie{{\em i.e.}, }
\def\eg{{\em e.g.} }

\def\bbm{\ensuremath{\mathbbm}}

\def\eps{\varepsilon}

\usepackage[margin=1in]{geometry}

\numberwithin{equation}{section}

\title{Almost sure convergence of cover times for $\psi$-mixing systems}
\date{December 2024}
\begin{document}

\author[B.~Zhao]{Boyuan Zhao}
\address{Mathematical Institute\\
University of St Andrews\\
North Haugh\\
St Andrews\\
KY16 9SS\\
Scotland} 
\email{bz29@st-andrews.ac.uk}

\begin{abstract}
Given a topologically transitive system on the unit interval, one can investigate the \emph{cover time}, \ie time for an orbit to reach certain level of resolution in the repeller. We introduce a new notion of dimension, namely the \emph{stretched Minkowski dimension}, and show that under mixing conditions, the asymptotics of typical cover times are determined by Minkowski dimensions  when they are finite, or by stretched Minkowski dimensions otherwise. For application, we show that for countably full-branched affine maps, results using the usual Minkowski dimensions fail to produce a finite log limit of cover times whilst the stretched version gives an finite limit. In addition, cover times of irrational rotations are explicitly calculated as counterexamples, due to the absence of mixing.
\end{abstract}
\maketitle

\section{Introduction}
Let $\X\subset[0,1]$ and $f:\X\to[0,1]$ a topologically transitive piecewise expanding Markov map equipped with an ergodic invariant probability measure $\mu$, we study the \emph{cover times} for points in the repeller $\Lambda$, \ie given $x\in\Lambda$ let
\[\tau_r(x):=\inf\left\{k:\forall y\in\Lambda, \exists j\le k:y\in d(f^j(x),y)<r\right\}.\]
The first quantitative result of expected cover times $\bbm E[\tau_r]$ was obtained for Brownian motions in \cite{Mat88}, and generalised in recent works \cite{BarJurKlo} and \cite{JurTod24} for \emph{chaos games} associated to iterative function system, and one dimensional dynamical systems respectively. In \cite{BarJurKlo}, an almost sure convergence for $-\log \tau_r/\log r$ was also demonstrated, assuming the invariant measure $\mu$ supported on the attractor of the IFS satisfies rapid mixing conditions. All results suggest that the asymptotic behaviour of $\tau_r$ is crucially linked to the Minkowski dimensions: let $M_\mu(r):=\min_{x\in\supp(\mu)}\mu(B(x,r))$, the \emph{upper and lower Minkowski dimensions} of $\mu$ are defined respectively by 
\begin{equation*}
\overline{\dim}_M(\mu):=\limsup_{r\to0}\frac{\log M_\mu(r)}{\log r},\hspace{2mm}\underline{\dim}_M(\mu):=\liminf_{r\to0}\frac{\log M_\mu(r)}{\log r},
\end{equation*} 
and simply write $\dim_M(\mu)$ when the two quantities coincide. In other words, these dimension-like quantities reflect the decay rate of the minimal $\mu$-measure ball at scale $r$, and they are closely related to the box-counting dimension of the ambient space (see \cite{FalFraKae} for more details). In addition, the Minkowski dimensions of $\mu$ govern the asymptotic behaviour of  hitting times associated to the balls which are most `unlikely' to be visited at small scales. Our first result below gives an almost sure asymptotic growth rate of cover times in terms of the Minkowski dimensions.
\begin{theorem}\label{thm: finite minkowski dimension case}
Let $(f,\mu)$ be a probability preserving system where $f$ is topologically transitive, Markov and piecewise expanding. If $\overline\dim_M(\mu)<\infty$, then for $\mu$-a.e. $x$ in the repeller,
\begin{equation*}\label{eqn: main thm finite Min dim}\limsup_{r\to0}\frac{\log \tau_r(x)}{-\log r}\ge\overline\dim_M(\mu),\hspace{3mm}\liminf_{r\to0}\frac{\log \tau_r(x)}{-\log r}\ge\underline{\dim}_M(\mu).
\end{equation*}
If $(f,\mu)$ is exponentially $\psi$-mixing, then for $\mu$-almost every $x\in\Lambda$, the inequalities above are improved to
\begin{equation*}
    \limsup_{r\to0}\frac{\log \tau_r(x)}{-\log r}=\overline\dim_M(\mu),\hspace{3mm}\liminf_{r\to0}\frac{\log \tau_r(x)}{-\log r}=\underline{\dim}_M(\mu).
\end{equation*}
\end{theorem}
 In particular, it is true if the invariant measure in question is \emph{doubling}.
\begin{remark}
   We remark that systems with finite Minkowski dimensions, or at least $\overline\dim_M(\mu)<\infty$, are fairly common. In particular, if $\mu$ is \emph{doubling}, \ie there exists constant $D>0$ such that for all $x\in\supp(\mu)$ and $r>0$, $D\mu(B(x,r))\ge \mu(B(x,2r))>0$, then $\overline\dim_M(\mu)<\infty$.
\end{remark}
\begin{proof}
    For each $n\in\N$ let $x_n\in\supp(\mu)$ be such that $\mu(B(x_n,2^{-n}))=M_\mu(2^{-n})$, then by the doubling property,
    \[M_\mu\left(2^{-n}\right)=\mu\left(B\left(x_n,2^{-n}\right)\right)\ge D^{-1}\mu\left(B\left(x_n,2^{-n+1}\right)\right)\ge D^{-1}M_\mu\left(2^{-n+1}\right)=D^{-1}\mu\left(B(x_{n-1}2^{-n+1}\right),\]
    and reiterating this one gets $M_\mu(2^{-n})\ge D^{-n+1}M_\mu(1/2)$, in other words
    \[ \frac{\log M_\mu(2^{-n})}{-n\log 2}\le \frac{-(n-1)\log D+\log M_\mu(1/2)}{-n\log 2}.\]
    As for all $r>0$, there is unique $n\in\N$ such that $2^{-n}<r\le 2^{-n+1}$, and $\frac{\log 2^{-n}}{\log 2^{-n+1}}=1$,
\[\limsup_{r\to0}\frac{\log M_\mu(r)}{\log r}=\limsup_{n\to\infty}\frac{\log M_\mu\left(2^{-n}\right)}{-n\log 2}\le \frac{\log D}{\log 2}<\infty.\qedhere\]
\end{proof}
However, the Minkowski dimensions are not always finite due to non-doubling behaviours, or more  extreme decay of $M_\mu(r)$ (see \Cref{Example: full-branched interval map}). Hence we need a new notion of dimension, invariant under scalar multiplication (replacing $M_\mu(r)$ by $M_\mu(cr)$ for any $c>0$ the limit does not change), to capture such decay rate in $r$.
\begin{definition}\label{def: generalised Minkowski dimension}
Define the upper and lower \emph{stretched Minkowski dimensions} by
    \[\overline{\dim}_M^s(\mu):=\limsup_{r\to0}\frac{\log|\log M_\mu(r)|}{-\log r},\hspace{3mm}\underline{\dim}_M^s(\mu):=\liminf_{r\to0}\frac{\log \log|M_\mu(r)|}{-\log r}.\]
\end{definition}

Those quantities should be of independent interest. Our second theorem below deals with almost sure cover times for systems in which $M_\mu(r)$ decays at stretched-exponential rates.
\begin{theorem}\label{thm: main theorem}
 Let $(f,\mu)$ be an ergodic probability preserving system where $f$ is topologically transitive, Markov and piecewise expanding. If $\overline\dim_M(\mu)=\infty$, but $0<\underline\dim_M^s(\mu),\overline\dim_M^s(\mu)<\infty$, then for $\mu$-almost every $x\in\Lambda$, 
 \begin{equation}\liminf_{r\to0}\frac{\log \log\tau_r(x)}{-\log r}\ge\underline{dim}_M^s(\mu),
 \hspace{3mm}\limsup_{r\to0}\frac{\log \log \tau_r(x)}{-\log r}\ge\overline{dim}_M^s(\mu)\label{eqn: main theorem without mixing}\end{equation}
 If $(f,\mu)$ is exponentially $\psi$-mixing, then for $\mu$-almost every $x\in\Lambda$,
 \begin{equation}\label{eqn: main theorem with mixing}
     \liminf_{r\to0}\frac{\log \log\tau_r(x)}{-\log r}=\underline{dim}_M^s(\mu), \hspace{3mm}\limsup_{r\to0}\frac{\log \log\tau_r(x)}{-\log r}=\overline{dim}_M^s(\mu).
 \end{equation}
\end{theorem}

\subsection*{Layout of the paper}
Basic definitions are introduced in \Cref{sec: setup} and we delay the proofs of the main theorems to \Cref{sec: proofs of main theorems}. Several examples that satisfy \Cref{thm: finite minkowski dimension case} and \Cref{thm: main theorem} will be discussed in \Cref{sec: examples}. In \Cref{sec: rotation} we will also prove that for irrational rotations, which are known to have no mixing behaviour, \Cref{thm: finite minkowski dimension case} fails for almost every point when the rotations are of type $\eta$ (see \Cref{def: rotation type}) for some $\eta>1$. Lastly in \Cref{section: flows} we show that similar results hold for flows under some natural conditions.

\section{Setup}\label{sec: setup}
Let $\mathcal{A}$ be a finite or countable index set, and $\mathcal{P}=\{P_a\}_{a\in\mathcal{A}}$ a collection of subintervals in $[0,1]$ with disjoint interiors covering $\X$. We say $f:\X\rightarrow[0,1]$ is a \textit{piecewise expanding Markov map} if\\
(1) for any $a\in\mathcal A$, $f_a:=f|_{P_a}$ is continuous, injective and $f(P_a)$ a union of elements in $\P$; \\
(2) there is a uniform constant $\gamma>1$ such that for all $a\in\mathcal A$, $|Df_a|\ge\gamma$. 

The \emph{repeller} of $f$, denoted by $\Lambda$, is the collection of points with all their forward iterates contained in $\P$, namely
    \[\Lambda:=\left\{x\in\X:f^k(x)\in\bigcup_{a\in\mathcal A}P_a\text{ for all $k\ge0$}\right\}.\]
We study the dynamics of $f:\Lambda\to\Lambda$, together with an ergodic invariant probability measure $\mu$ supported on $\Lambda$.
There is a shift system associated to $f$: let $M$ be an $\mathcal{A}\times \mathcal{A}$ matrix such that $M_{ab}=1$ if $f(P_a)\cap P_b\neq \emptyset$ and 0 otherwise. $f$ is \emph{topologically transitive} if for all $a,b\in\mathcal A$, there exists $k$ such that $M^k_{ab}>0$. Let $\Sigma$ denote the space of all \emph{infinite admissible words}, \ie
\[\Sigma:=\left\{{x}=(x_0,x_1,\dots)\in\mathcal A^{\N_0}:M_{x_k,x_{k+1}}=1,\,\forall \,k\ge0\right\}.\]
A natural choice of metric on $\Sigma$ is $d_s(x,y):=2^{-\inf\{j\ge0:\,x_j\neq y_j\}}$, and we define the projection map $\pi:\Sigma\to\Lambda$ by
$${x}=\pi\left(x_0,x_1,\dots\right) \text{ if and only if } x\in\bigcap_{i=0}^{\infty}f^{-i}P_{x_i}.$$ 
The dynamics on $\Sigma$ is the left shift $\sigma:\Sigma\to\Sigma$ given by $\sigma(x_0,x_1,\dots,)=(x_1,x_2,\dots)$, then $\pi$ defines a semi-conjugacy $f\circ\pi=\pi\circ\sigma$, and the corresponding symbolic measure $\tilde\mu$ of $\mu$ is given by $\mu=\pi_*\Tilde\mu$, \ie for all Borel-measurable set $B\in\mathcal B([0,1])$, $\mu(B)=\tilde\mu\left(\pi^{-1}B\right)$.

Denote $\mathcal{P}^n:=\bigvee_{j=0}^{n-1}f^{-j}\mathcal{P}$, each $P\in\mathcal P^n$ corresponds to an \emph{$n-$cylinder} in $\Sigma$: let $\Sigma_n\subseteq \mathcal A^n$ denote all finite words of length $n$ and for any $\textbf{i}\in\Sigma$, the $n$-cylinder defined by $\textbf{i}$ is 
\[[\textbf{i}]=[i_0,\dots,i_{n-1}]:=\left\{y\in\Sigma:y_j=i_j,\,j=0,\dots,n-1\right\},\]
then $\pi[i_0,i_1,\dots,i_{n-1}]=\bigcap_{j=0}^{n-1}f^{-j}P_{i_j}=:P_{\textbf{i}}$.

Furthermore, $(f,\mu)$ is required to have the following mixing property.
\begin{definition}\label{def: psi mixing}
    Say $\mu$ is \emph{exponentially $\psi$-mixing} if there are $C_1,\rho>0$ and a monotone decreasing function $\psi(k)\le C_1e^{-\rho k}$ for all $k\in\N$, such that the corresponding symbolic measure $\Tilde\mu$ satisfies: for all $n,k\in \N$, $\textbf{i}\in\Sigma_n$ and $\textbf{j}\in\Sigma^*=\bigcup_{l\ge1}\Sigma_l$,
    \[\left|\frac{\Tilde\mu([\textbf{i}]\cap\sigma^{-(n+k)}[\textbf{j}])}{\Tilde\mu[\textbf{i}]\Tilde\mu[\textbf{j}]}-1\right|\le \psi(k).\]
\end{definition}

\section{Examples}\label{sec: examples}
\Cref{thm: main theorem} is applicable to the following systems.
\begin{example}
    Finitely branched Gibbs-Markov maps: let $f$ be a topologically transitive piecewise expanding Markov map with $\mathcal A$ finite. $f$ is said to be \emph{Gibbs-Markov} if for some \emph{potential} $\phi:\Sigma\to\R$ which is \emph{locally H\"older} with respect to the symbolic metric $d_s$, there exists $G>0$ and $P\in\R$ such that for all $n\in\N$, all $x=(x_0,x_1,\dots)\in\Sigma$,
    \[\frac1G\le \frac{\tilde\mu([x_0,\dots,x_{n-1}])}{\exp\left(\sum_{j=0}^{n-1}\phi(\sigma^jx)-nP\right)}\le G.\]
For maps of this kind, $|Df|$ is uniformly bounded so for each ball at scale $r$, it is possible to approximate any ball with finitely many cylinders of the same depth (see for example the proof of \cite[Lemma 3.2]{JurTod24}), and by the Gibbs property of $\tilde\mu$, the asymptotic decay rate converges so $\dim_M(\mu)$ exists and is finite. Since Gibbs measures are exponentially $\psi$-mixing (see \cite[Proposition 1.14]{Bow75}), by \Cref{thm: finite minkowski dimension case}, we have
    \[\lim_{r\to0}\frac{\log\tau_r(x)}{-\log r}=\dim_M(\mu)\]
for $\mu$-a.e. $x$ in the repeller of $f$.
\end{example}

In the next example, when $r\to0$ at polynomial rate, $M_\mu(r)$ decays exponentially hence $\overline\dim_M(\mu)$ is infinite, and the stretched Minkowski dimensions are needed.
\begin{example}\label{Example: full-branched interval map}
Similar to \cite[Example 7.4]{JurTod24}, consider the following class of infinitely full-branched maps: pick $\kappa>1$ and set $c=\zeta(\kappa)=\sum_{n\in\N}\frac1{n^{\kappa}}$. Let $a_0=0$, $a_j=\sum_{j=1}^n\frac1{cj^{\kappa}}$ and define $f$ by
    \begin{equation*}
    \forall\,n\in\N_0=\N\cup\{0\},\,   f(x)=cn^\kappa(x-a_{n-1})\text{    for $x\in[a_{n-1},a_{n})=:P_{n}$.}
    \end{equation*}
    Then $f$ is an infinitely full-branched affine map, and we can associate this map with a full-shift system on $\N$: $x=\pi(i_0,i_1,\dots)$ if for all $j\ge1$, $f^j(x)\in P_{i_j}$. \\
    Let $\omega>1$ and construct $\tilde\mu$ the finite Bernoulli measure by
    \[\tilde\mu([i_0,\dots,i_{n-1}])=\prod_{j=0}^{n-1}\omega^{-i_j},\]
    so the push-forward measure $\mu=\pi_*\tilde\mu$ has $\mu(P_n)=\omega^{-n}$.
    \begin{proposition}
        $\overline{\dim}_M(\mu)=\infty$, but $\dim_M^s(\mu)=\frac1{\kappa-1}$.
    \end{proposition}
    \begin{proof}
            For each $r>0$, $M_\mu(r)$ is found near $1$, then along the sequence $r_n=\frac{1}{2c}\sum_{j\ge n}{j^{-\kappa}}\approx\frac{1}{2c(\kappa-1)n^{\kappa-1}}$, the ball that realises $M_\mu(r_n)$ is contained in $\bigcup_{j=n}^\infty P_j$, hence 
            \[\omega^{-n}\le M_\mu(r_n)\le \frac{\omega^{-n}}{1-\omega^{-1}},\]  
            therefore 
            \[\overline{\dim}_M(\mu)\ge\limsup_{n\to\infty}\frac{n\log \omega}{(\kappa-1)\log n} =\infty,\]
            whereas for all $n$,
            \[\frac{\log n}{(\kappa-1)\log n}\le\frac{\log|\log M_\mu(r_n)|}{-\log r_n}\le \frac{\log n+\log\log \omega}{(\kappa-1)\log n-\log (2c(\kappa-1))}.\]
            As for all $r>0$, there is unique $n\in\N$ such that $r_{n+1}\le r <r_n$ while $\lim_{n\to\infty}\frac{\log r_{n+1}}{\log r_n}=1$, one concludes with $\dim^{s}_M(\mu)=\frac{1}{\kappa-1}$.\qedhere
    \end{proof}
As in \cite[Example 7.4]{JurTod24} it is very difficult for the system to cover small neighbourhoods of 1 so \Cref{thm: finite minkowski dimension case} says $\limsup_{r\to0}\frac{\log \tau_r(x)}{-\log r}\ge\overline\dim_M(\mu)=\infty$, but since $\tilde\mu$ is Bernoulli hence $\psi$-mixing, \Cref{thm: main theorem} asserts that
        \[\lim_{r\to0}\frac{\log\log\tau_r(x)}{-\log r}=\frac1{\kappa-1}\text{\hspace{3mm}$\mu$-a.e.}\]
\end{example}

\section{Proof of \Cref{thm: main theorem}}\label{sec: proofs of main theorems}
The proofs in this section are adapted from those of \cite[Proposition 3.1, 3.2]{BarJurKlo}. We will only demonstrate the proofs for \Cref{thm: main theorem}, \ie the asymptotics are determined by stretched Minkowski dimensions; the proofs for \Cref{thm: finite minkowski dimension case} are obtained by replacing all stretched exponential sequences in the proofs below by some exponential sequence, \eg for a given constant $s\in\R$, $e^{\pm n^{s}}$ will be replaced by $2^{\pm ns}$. 

Assuming the inequalities in \eqref{eqn: main theorem without mixing}, we first prove the \eqref{eqn: main theorem with mixing} which requires the exponentially $\psi$-mixing condition.

\begin{remark}\label{remark subsequence trick}
    Assuming the conditions of \Cref{thm: main theorem}, we will prove that the statements hold along the subsequence $r_n=n^{-1}$ such that for each $r>0$ there is a unique $n\in\N$ with $r_{n+1}<r\le r_n$ while $\lim_{n\to\infty}\frac{\log r_{n+1}}{\log r_n}=1$ (if $\overline{\dim}_M(\mu)$ or $\underline{\dim}_M(\mu)$ are finite we choose $r_n=2^{-n}$ instead), and since $\log\tau_{r}(x)$ is increasing as $r\to0$,
    \[\limsup_{n\to\infty}\frac{\log\log\tau_{r_n}(x)}{-\log r_n}=\limsup_{r\to0}\frac{\log\log\tau_r(x)}{-\log r},\]and similarly for liminf's.
\end{remark}

\subsection{Proof of \eqref{eqn: main theorem with mixing}}
\begin{proposition}\label{Prop upper bound 1}
    Suppose $(f,\mu)$ is exponentially $\psi$-mixing, and the upper stretched Minkowski dimension of $\mu$, $\overline\dim_M^s(\mu)$, is finite, 
    then for $\mu$-almost every $x\in\Lambda$, 
    \[\limsup_{n\rightarrow\infty}\frac{\log \log\tau_{r}(x)}{-\log r}\le \overline\dim^s_M(\mu).\]
\end{proposition}
\begin{proof}
Let $\eps>0$, and for simplicity denote $\overline\alpha:=\overline{\dim}^s_M(\mu)$. 

For any finite $k$-word $\textbf{i}=x_0,\dots,x_{k-1}\in\Sigma_k$, let $\textbf{i}^-=x_0,\dots,x_{k-2}$, \ie $\textbf{i}$ dropping the last digit. Recall that for each $\textbf{i}\in\Sigma^*$, $P_\textbf{i}=\pi[\textbf{i}]$, and we define 
\[\mathcal W_r:=\{\textbf{i}\in\Sigma^*:\,\text{diam}(P_{\textbf{i}})\le r<\text{diam}(P_{\textbf{i}^-})\}.\]
By expansion, for each $n\in\N$, the lengths of the words in $\mathcal W_{n^{-1}}$ are bounded from above, hence we can define
$$L(n):=\frac{\log n}{\log\gamma}+1
\ge\max\{|\textbf{i}|:\textbf{i}\in \mathcal W_{n^{-1}}\}.$$ 
Given $y\in[0,1]$ and $r>0$ such that $B(y,r)\subset \supp(\mu)$, define the corresponding symbolic balls by \[\Tilde B(y,r):=\left\{[\textbf{i}]:\textbf{i}\in \mathcal W_r,P_{\textbf{i}}\cap B(y,r)\neq \emptyset\right\}.\] 
Note that if for some $x\in P_{\textbf{i}}$ and $[\textbf{i}]\in\Tilde{B}(y,r)$, $d(x,y)\le r+diam(P_\textbf{i})\le 2r$ therefore $$B(y,r)\subset \pi \Tilde B(y,r)\subset B(y,2r).$$
    
Let $\Q_{n}$ be a cover of $\Lambda$ with balls of radius $r_n=1/2n$, denote the collection of their centres by $\mathcal Y_n$, and  $\#\Q_n=\#\mathcal Y_n\le n$. Let $\tau(\Q_n,x)$ be the minimum time for the orbit of $x$ to have visited each element of $\Q_n$ at least once, 
\[\tau(\Q_n,x):=\min\left\{k\in\N:\text{ for all }Q\in\Q_n,\text{ exists }0\le j\le k: f^j(x)\in Q\right\}.\]
Then $\tau_{1/n}(x)\le \tau(\Q_n,x)$ for all $n$ and all $x$ since for all $y\in\Lambda,$ there is $Q\in\Q_n$ and $j\le \tau(\Q_n,x)$ such that $f^j(x),y\in Q$ hence $d(f^j(x),y)\le 1/n.$  Let $\eps>0$ be an arbitrary number and for each $k\in\N$, set $L'(k)=L(k)+\frac1{\rho}\left(k^{\overline\alpha+\eps}-\log C_1\right)$ where $C_1$, $\rho$ were given in \Cref{def: psi mixing}, then
    \begin{equation}
        \begin{split}\label{eqn: upper bound sum 1}
        &\mu\left(x:\tau_{1/n}(x)>e^{n^{\overline\alpha+\eps}}L'(4n)\right)\le \mu\left(x:\tau(\Q_n,x)>e^{n^{\overline\alpha+\eps}}L'(4n)\right)\\
        &=\mu\left(x:\,\exists y\in\mathcal Y_n:f^j(x)\not\in B(y,{1}/{2n}),\,\forall j\le e^{n^{\overline\alpha+\eps}}L'({2n})\right)\\
        &\le \mu\left(x:\exists y\in\mathcal Y_n:f^{jL'(4n)}(x)\not \in B(y,1/{2n}),\,\forall j\le e^{n^{\overline\alpha+\eps}}\right)\\
        &=\mu\left(\bigcup_{y\in\mathcal Y_n}\bigcap_{j=1}^{e^{n^{\overline\alpha+\eps}}}f^{-jL'(4n)}(x)\not\in B(y,1/2n)\right)\le \sum_{y\in\mathcal Y_n}\mu\left(\bigcap_{j=1}^{e^{n^{\overline\alpha+\eps}}}f^{-jL'(
    2n)}(x)\not\in B(y,1/2n)\right).
    \end{split}
    \end{equation}
As $\pi\left(\tilde B(z,r)\right)\subseteq B(z,2r)$ for all $z$ and all $r>0$, using the exponentially $\psi$-mixing property of $\tilde\mu$, \ie $\psi(k)\le C_1e^{-\rho k}$ for all $k\in\N$, by our choice of $L'(4n)$,
\begin{equation}
\begin{split}\label{eqn: upper bound sum 2}
    &\sum_{y\in\mathcal Y_n}\mu\left(x:\bigcap_{j=1}^{e^{n^{\overline\alpha+\eps}}}f^{-jL'(4n)}(x)\not\in B(y,1/2n)\right)\le \sum_{y\in\mathcal Y_{k+1}}\Tilde\mu\left(x:\bigcap_{j=1}^{e^{n^{\overline\alpha+\eps}}}\sigma^{-jL'(4n)}(\pi^{-1}x)\not\in \Tilde B\left(y,1/4n\right)\right)\\
    &\le \left(1+\psi\left(\frac1{\rho}\left((4n)^{\overline\alpha+\eps}-\log C_1\right)\right)\right)^{e^{n^{\overline\alpha+\eps}}}\sum_{y\in\mathcal{Y}_{k+1}}\left(1-\Tilde\mu\left(\Tilde B\left(y,\frac1{2n}\right)\right)\right)^{e^{n^{\overline\alpha+\eps}}}\\
    &\le \left(1+e^{-n^{\overline\alpha+\eps}}\right)^{e^{n^{\overline\alpha+\eps}}}\sum_{y\in\mathcal{Y}_{k+1}}\left(1-\mu\left(B\left(y,\frac{1}{4n}\right)\right)\right)^{e^{n^{\overline\alpha+\eps}}}.
\end{split}
\end{equation}
    By definition of $\overline\alpha$, for all $n$ large such that $\frac\eps4\log n\ge \left(\overline\alpha+\frac\eps4\right)\log 4$, there is
    $$\log\left(-\log M_\mu\left(\frac1{4n}\right)\right)\le \overline{\alpha}+\eps/4(\log 4n)\le (\overline\alpha+\eps/2)\log n,$$
    so for all $y\in\supp(\mu)$ and all $n$ large,
    \[\mu\left(B\left(y,\frac1{4n}\right)\right)\ge e^{-n^{\overline\alpha+\eps/2}}\ge \frac{e^{-n^{\eps/2}}}{e^{n^{\overline\alpha+\eps}}}.\]
As for all $u\in\R$ and large $k$, $(1+\frac uk)^{k}\approx e^u$, combining \eqref{eqn: upper bound sum 1} and \eqref{eqn: upper bound sum 2}, for some uniform constant $C_2>0$,
\begin{align*}
    &\mu\left(x:\tau_{1/n}(x)>e^{n^{\overline\alpha+\eps}}L'(4n)\right)\le 
        \left(1+e^{-n^{\overline\alpha+\eps}}\right)^{e^{n^{\overline\alpha+\eps}}}\sum_{y\in\mathcal{Y}_{k+1}}\left(1- e^{-n^{\overline\alpha+\eps/2}}\right)^{e^{n^{\overline\alpha+\eps}}}\\
    &\le \left(1+e^{-n^{\overline\alpha+\eps}}\right)^{e^{n^{\overline\alpha+\eps}}}n\left(1-\frac{e^{n^{\eps/2}}}{e^{n^{\overline\alpha+\eps}}}\right)^{e^{n^{\overline\alpha+\eps}}}\le C_2\exp\left(\log n-{e^{n^{\eps/2}}}\right),
\end{align*}
which is clearly summable over $n$. Then by Borel Cantelli, for all $n$ large enough $\tau_{1/n}(x)\le e^{n^{\overline\alpha+\eps}}L'(4n)$. Since $\log L'(4n)\approx (\overline\alpha+\eps)\log n\ll n^{\overline\alpha+\eps}$, we have for $\mu-$a.e. $x\in\Lambda$,
\[\limsup_{n\to\infty}\frac{\log\log\tau_{1/n}(x)}{\log n}\le\limsup_{n\to\infty}\frac{\log\log \left(e^{n^{\overline\alpha+\eps}}L'(4n)\right)}{\log n}\le \overline{\alpha}+\eps.\] 
By \Cref{remark subsequence trick} this upper bound for $\limsup$ holds for all sequences decreasing to 0 and as $\eps>0$ was arbitrary, sending it to 0 one obtains that for $\mu-$a.e. $x\in\Lambda$,
\[\limsup_{r\to0}\frac{\log\log\tau_r(x)}{-\log r}=\limsup_{n\to\infty}\frac{\log\log\tau_{1/n}}{\log n}\le \overline{\alpha}.\qedhere\]
\end{proof}

\begin{proposition}\label{Prop lower bound 1}
 Suppose $(f,\mu)$ is exponentially $\psi$-mixing and the lower stretched Minkowski dimension of $\mu$, $\underline\dim_M^s(\mu)$, is finite, 
then for $\mu$-a.e. $x\in\Lambda$,
    \[\liminf_{r\to0}\frac{\log\log \tau_r(x)}{-\log r}\le \underline{\dim}_M^s(\mu).\]

\end{proposition}
\begin{proof}
Again for simplicity, denote $\underline\alpha:=\underline{dim}^s_M(\mu)$.
Let $\eps>0$ and by definition of liminf there is a subsequence $\{n_k\}_k\to\infty$ such that for all $k$,
\[\frac{\log (-\log M_\mu(1/n_k))}{\log n_k }\le \underline{\alpha}+\eps,\]
then repeating the proof of \Cref{Prop upper bound 1} by replacing $n$ by $n_k$  everywhere, one gets that for $\mu-$almost every $x$, 
\[\liminf_{k\to\infty}\frac{\log\log\tau_{1/n_k}(x)}{\log n_k}\le\underline{\alpha}+\eps.\]
Again send $\eps\to0$, and use the fact that liminf over the entire sequence is no greater than the liminf along any subsequence, the proposition is proved.
\end{proof}


\subsection{Proof of the inequalities \eqref{eqn: main theorem without mixing}}
\begin{proposition}\label{Prop lower bound 2}
For $\mu$-almost every $x\in\Lambda$,
    \[\liminf_{n\to\infty}\frac{\log\log \tau_{r}(x)}{-\log r}\ge\underline\dim^s_M(\mu).\]
\end{proposition}

\begin{proof}
We continue to use the notation $\underline\alpha=\underline\dim^s_M(\mu)$.
Let $\eps>0$ be arbitrary and by definition of $\underline\alpha$ for all large $n$ there exists $y_n\in \supp (\mu)$ such that $\mu(B(y_n,1/n))\le e^{-n^{\overline\alpha-\eps}}$. Let 
\[T(x,y,r):=\inf\left\{j\ge0:f^j(x)\in B(y,r)\right\},\]
so for all $n\in\N$ and all $x$, $\tau_{1/n}(x)\ge T(x,y_n,1/n)$. Then by invariance,
\begin{align*}
    &\mu\left(x:\tau_{1/n}(x)<  e^{n^{\overline\alpha-\eps}}/n^2\right)\le \mu\left(x:T(x,y_n,1/n)<   e^{n^{\overline\alpha-\eps}}/n^2\right)\\
    &=\mu\left(x:\exists\,0\le j<  e^{n^{\overline\alpha-\eps}}/n^2:\,f^j(x)\in B(y_n,1/n)\right)\le \bigcup_{j=0}^{ e^{n^{\overline\alpha-\eps}}/n^2}\mu\left(x:f^j(x)\in B(y_n,1/n)\right)\\
    &\le \sum_{j=0}^{e^{n^{\overline\alpha-\eps}}/n^2}\mu\left(f^{-j}B\left(y_n,\frac1n\right)\right)\le\frac{ e^{n^{\overline\alpha-\eps}}}{n^2} e^{-n^{\overline\alpha-\eps}}=\frac1{n^2},
\end{align*}
which is summable. By Borel-Cantelli, since $2\log n\ll n^{\underline \alpha-\eps}$, for $\mu$-almost every $x$
\[\liminf_{n\to\infty}\frac{\log\log \tau_{1/n}(x)}{\log n}\ge\underline\alpha-\eps,\]
and since $\eps>0$ was arbitrary one can send it to 0.\qedhere
\end{proof} 

Similar to \Cref{Prop upper bound 1} and \Cref{Prop lower bound 1},
\begin{proposition}
    For $\mu$-almost every $x\in\Lambda$,
    \[\limsup_{r\to0}\frac{\log\log \tau_r(x)}{-\log r}\ge\overline{\dim}_M^s(\mu).\]
\end{proposition}
\begin{proof}
    Let $\eps>0$, then by definition of limsup there exists a subsequence $\{n_k\}_k\to\infty$ such that for all $k$,
    \[\frac{\log \log\left(-M_\mu(1/n_k)\right)}{\log n_k}\ge\overline{\alpha}-\eps.\]
    Then repeating the proof of \Cref{Prop lower bound 2} along $\{n_k\}_k$, one gets that for $\mu$-almost every $x$: 
    \[\limsup_{k\to\infty}\frac{\log\log\tau_{1/n_k}(x)}{\log n_k}\ge \overline{\alpha}-\eps,\]
    then sending $\eps\to0$,
    \[\limsup_{r\to0}\frac{\log\log\tau_r(x)}{-\log r}\ge\limsup_{k\to\infty}\frac{\log\log\tau_{1/n_k}(x)}{\log n_k}\ge\overline{\alpha}.\qedhere\]
\end{proof}

\section{Irrational rotations}\label{sec: rotation}
    The proof of \eqref{eqn: main theorem with mixing} requires an exponentially $\psi$-mixing rate which is a strong mixing condition, and it is natural to ask if the same asymptotic growth in \Cref{thm: main theorem} remains the same under different mixing conditions, e.g. exponentially $\phi$-mixing and $\alpha$-mixing, or even polynomial $\psi$-mixing. Although these questions are unresolved, in this section we will show that the limsup and liminf of the asymptotic growth rate can differ if the system is not mixing at all.  
    
    Let $\theta\in(0,1)$ be an irrational number and $T(x)=x+\theta\text{ (mod 1)}$, and $\mu$ the one-dimensional Lebesgue measure on $[0,1)$. Then $(T,\mu)$ is an ergodic probability preserving system with $\dim_M(\mu)=1$. 
    \begin{definition}\label{def: rotation type}
        For a given irrational number $\theta$, the \emph{type of} $T_\theta$ is given by the following number
        \[\eta=\eta(\theta):=\sup\left\{\beta:\liminf_{n\to\infty}n^{\beta}\|n\theta\|=0\right\},\]where for every $r\in\R$, $\|r\|=\min_{n\in\bbm Z}|r-n|$. 
    \end{definition}
\begin{remark} (See \cite{Khi64})
    For every $\theta\in(0,1)$ irrational, $\eta(\theta)\ge1$ and $\eta(\theta)=1$ almost everywhere, but there exists irrational number with $\eta(\theta)\in(1,\infty]$, e.g. the Liouville numbers.
\end{remark}
   
        For any irrational number $\theta\in(0,1)$ there is a unique continued fraction expansion
        \[\theta=[a_1,a_2,\dots]:=\frac{1}{a_1+\frac{1}{a_2+...}}\]
        where $a_i\ge 1$ for all $i\ge1$. Set $p_0=0$ and $q_0=1$, and for $i\ge 1$ choose $p_i,q_i\in\N$ coprime such that 
        \[\frac{p_i}{q_i}=[a_1,\dots,a_i]=\frac1{a_1+\frac1{\dots\frac1{a_i}}}.\]
     \begin{definition}   
       The $a_i$ terms are called the \emph{$i$-th partial quotient} and $p_i/q_i$ the $i$-th \emph{convergent}. In particular, (see \cite{Khi64})\[\eta(\theta)=\limsup_{n\to\infty}\frac{\log q_{n+1}}{\log q_n}.\]
    \end{definition}

\begin{theorem}
For any irrational rotation $T_\theta$ with $\eta(\theta)>1$, 
\[\liminf_{r\to0}\frac{\log \tau_r(x)}{-\log r}=\dim_M(\mu)=1<\eta=\limsup_{r\to0}\frac{\log \tau_r(x)}{-\log r} \text{ $\mu$-a.e.}\]
\end{theorem}
The proof of this theorem relies on algebraic properties of the number $\eta(\theta)$, and for simplicity write $\eta$ from now on.
    \begin{lemma}\cite[Fact 1, Lemma 7]{KimSeo03}\label{lemma: convergent-denominator}
        \begin{enumerate}[(a)]
        \item $q_{i+2}=a_{i+2}q_{i+1}+q_i$ and $p_{i+2}=a_{i+2}p_{i+1}+p_i$.
            \item $1/{\left(2q_{i+1}\right)}\le1/\left({q_{i+1}+q_i}\right)<\|q_i\theta\|<1/q_{i+1}$ for $i\ge1$.
            \item If $0<j<q_{i+1}$, then $\|j\theta\|\ge \|q_i\theta\|$.
         \item for $\eps>0$, there exists uniform $C_{\eps}>0$ such that for all $j\in\N$, $j^{\eta+\eps}\|j\theta\|>C_\eps$.
        \end{enumerate}
    \end{lemma}
The following propositions are largely based on \cite[Proposition 6, Proposition 10]{KimSeo03}.
\begin{proposition}For $\mu$-a.e. $x$,
    \begin{equation}
        \limsup_{r\to0}\frac{\log \tau_r(x)}{-\log r}\ge \eta.
    \end{equation}
\end{proposition}
\begin{proof}First we prove the following simple claim.
\begin{claim}
        The function $\varphi(x)=\limsup_{r\to0}\frac{\log \tau_r(x)}{-\log r}$ is constant $\mu$ a.e.
    \end{claim} 
    \begin{proof}[Proof of Claim.]
        Suppose $\tau_r(x)=k$, and for any $y\neq x$, if there exists $z$ such that for all $0\le j\le k,$ $\left|T^jy-z\right|\ge r$, then for all $0\le j\le k$
        \[\left|T^jy-x+x-z\right|=\left|T^jx-(x+z-y)\right|\ge r,\]
        contradicting $\tau_r(x)=k$, then by symmetry $\tau_r(x)=\tau_r(y)$, in particular $\tau_r(x)=\tau_r(Tx)$ so $\varphi\circ T=\varphi$, and $\mu$ is (uniquely) ergodic implies $\varphi$ is constant almost everywhere. \qedhere
    \end{proof}
    
    By \cite[Proposition 10]{KimSeo03}, for almost every $x,y$
    \[\limsup_{r\to0}\frac{\log W_{B(y,r)}(x)}{-\log r}\ge \eta,\]
    where $W_E(x):=\inf\{n\ge 1:T^nx\in E\}$ denotes the waiting time of $x$ before visiting $E$.
    Hence there exists a set of strictly positive measure consisting of points that satisfy
    \[\limsup_{r\to0}\frac{\log \tau_r(x)}{-\log r}\ge \limsup_{r\to0}\frac{\log W_{B(y,r)}(x)}{-\log r}\ge \eta,\]
    since for all $y\in[0,1)$, $\tau_r(x)\ge W_{B(y,r)}(x)$. By the claim above this holds for almost every $x$.
\end{proof}
\begin{proposition}\label{prop: upper bound for limsup rotation}
    For $\mu$-a.e. $x$,
    \[\limsup_{r\to0}\frac{\log \tau_r(x)}{-\log r}\le \eta.\]
\end{proposition}
\begin{proof}
Let $\Q_n:=\left\{\left[2^{-n}j,2^{-n}(j+1)\right):j=0,\dots,2^n-1\right\}$ and $\tau\left({\Q_n},x\right)$ the minimum time for $x$ to have visited each element of $\Q_n$, again we have $\tau_{2^{-n+1}}(x)\le \tau(\Q_n,x)$ for all $x$. By \Cref{lemma: convergent-denominator} (a) and (c), $\{\|q_i\theta\|\}_i$ is a decreasing sequence, and there for each $n\in\N$ exists a minimal $j$ such that $\|q_{j}\theta\|< 2^{-n}\le \|q_{j-1}\theta\|$, write $j=j_n$.

By \cite[Proposition 6]{KimSeo03} for all $n$, there is $\mu\left(W_{[0,2^{-n})}>q_{j_n}+q_{j_n-1}\right)=0$. Notice that for all $a,b\in[0,1)$,
    \begin{equation}\label{eqn: Leb translation invariant}
        \mu\{W_{[a,a+b)}(x)=k\}=\mu\left\{\{x:W_{[0,b)}(x)=k\}+a\right\}=\mu\left\{W_{[0,b)}(x)=k\right\},
    \end{equation}
    as $\mu=Leb$ is translation invariant. Then by \eqref{eqn: Leb translation invariant}
    \begin{equation*}\label{eqn: 0 measure tau_Q}
    \begin{split}
        &\mu\left\{\tau\left({\Q_n},x\right)>q_{j_n}+q_{j_n-1}\right\}=\mu\left\{x:\,\forall Q\in\Q_{n}:\,W_Q(x)>q_{j_n}+q_{j_n-1}\right\}\\
        &=\mu\left(x:\bigcup_{Q\in\Q_{n}}\left\{W_Q(x)>q_{j_n-1}+q_{j_n}\right\}\right)\le \sum_{Q\in\Q_{n}}\mu\left(W_Q>q_{j_n-1}+q_{j_n}\right)\\
        &=\sum_{j=0}^{2^n-1}\mu\left(W_{\left[2^{-n}j,2^{-n}(j+1)\right)}>q_{j_n}+q_{j_n-1}\right)=\sum_{j=0}^{2^n-1}\mu\left(W_{[0,2^{-n})}>q_{j_n}+q_{j_n-1}\right)=0.
    \end{split}
    \end{equation*}
    Hence by Borel-Cantelli, for all $n$ large enough, $\tau_{2^{-n+1}}(x)\le (q_{j_n}+q_{j_n-1})$ for $\mu$-a.e $x\in [0,1)$. 
    
    Let $\eps>0$, and by \Cref{lemma: convergent-denominator} there exists $C_\eps$ such that 
    \begin{equation*}
        \log\left(q_{j_n}+q_{j_n-1}\right)\le \log\left( 2q_{j_n}\right)\le\log\frac2{\|q_{j_n}\theta\|}\le (\eta+\eps)\log q_{j_n}+\log2-\log C_\eps.
    \end{equation*}
    Again by \Cref{lemma: convergent-denominator} and our choice of $j_n$, for all $n$ large enough, up to a uniform constant 
    \begin{equation*}
        \log\tau_{2^{-n+1}}(x)\le \log(q_{j_n}+q_{j_n-1})<(\eta+\eps)\log q_{j_n}\le-(\eta+\eps)\log\|q_{j_n-1}\theta\|\le (\eta+\eps)n\log 2.
    \end{equation*}
    Hence $\limsup_{n\to\infty}\frac{\log\tau_{2^{-n}}(x)}{n\log2}\le \eta+\eps$ for $\mu$-almost every $x$, and send $\eps$ to 0 the proposition is proved since for each $r<0$ there is a unique $n\in\N$ for which $2^{-n}<r\le 2^{-n+1}$.
\end{proof}

\begin{proposition}
    For $\mu$-almost every $x\in[0,1)$, 
    \[\liminf_{r\to0}\frac{\log \tau_r(x)}{-\log r}= 1.\]
\end{proposition}
\begin{proof}
Let $\eps>0$, and using the same arguments in the last proof, \ie cover time is greater than the hitting time of the ball of smallest measure at scale $r$, then along the sequence $r_n=2^{-(n+1)}$, one gets for all $[a-r_n,a+r_n)\subset[0,1)$,  there is 
\begin{align*}
&\sum_{n\ge1}\mu\left(\tau_{r_n}(x)< 2^{n(1-\eps)}\right)\le \sum_{n\ge1}\mu\left(W_{[a-2^{-n-1},a+2^{-n-1})}(x)< 2^{n(1-\eps)}\right)\\
&\le \sum_{n\ge1}\sum_{k=0}^{2^{n(1-\eps)}}\mu\left(T^{-k}[a-2^{-n-1},a+2^{-n-1})\right)=\sum_{n\ge1}2^{n(1-\eps)}2^{-n}=\sum_{n\ge1}2^{-\eps n}<\infty.
\end{align*}
Since for each $r$ there is a unique $n$ such that $r_n<r \le r_{n-1}$ while $\lim_n\frac{\log r_n}{\log r_{n-1}}=1$ so by Borel Cantelli, 
\[\liminf_{r\to0}\frac{\log\tau_r(x)}{-\log r}=\liminf_{n\to\infty}\frac{\log \tau_{2^{-n}}(x)}{n\log 2}\ge1-\eps,\]and sending $\eps$ to 0 the proposition is proved.

    For the upper bound of liminf, recall that $\tau\left({\Q_n},x\right)\ge \tau_{2^{-n}}(x)$, we can repeat the proof of \Cref{prop: upper bound for limsup rotation}, apart from that this time we choose $\{2^{-n_i}\}_i$ according to $\{q_i\}_{i\in\N}$: for each $i$, choose $n_i\in\N$ to be the smallest number such that 
    \[\|q_{i+1}\theta\|< 2^{-n_i}\le \|q_{i}\theta\|,\]
   hence as in \Cref{prop: upper bound for limsup rotation},
    \begin{equation*}
    \mu\left(\tau\left({\Q_{n_i}},x\right)>q_{i+1}+q_i\right)\le \sum_{Q\in\Q_{n_i}}\mu\left(W_Q>q_{i+1}+q_i\right)=0.
    \end{equation*}
Again by \Cref{lemma: convergent-denominator} (b), $q_{i+1}+q_i\le 2q_{i+1}\le\frac2{\|q_i\theta\|}<2^{n_i+1}$ by our choice of $n_i$, so $\lim_{i\to\infty}\frac{\log (q_i+q_{i+1})}{n_i\log2}\le 1$, therefore for $\mu$-a.e. $x$,
\[\liminf_{r\to0}\frac{\log\tau_r(x)}{-\log r}\le \liminf_{i\to\infty}\frac{\log \tau_{2^{-n_i}}(x)}{n_i\log2}\le \liminf_{i\to\infty}\frac{\log \tau\left({\Q_{n_i}},x\right)}{n_i\log2}\le 1.\qedhere\]
\end{proof}

\section{Cover time for flows}\label{section: flows}
In this section we prove results analogous to \Cref{thm: finite minkowski dimension case} regarding cover times for the same class of flows discussed in \cite[\S4]{RouTod24}. 

Let $\{f_t\}_t$ be a flow on a metric space $(\X,d_\X)$ preserving an ergodic probability measure $\nu$, \ie $\nu\left(f_t^{-1}A\right)=\nu(A)$ for every $t\ge0$ and $A$ measurable. Let $\Omega$ denote the non-wandering set and define the cover time of $x$ at scale $r$ by
\[\tau_r(x):=\inf\left\{T>0:\,\forall y\in\Omega, \,\exists t\le T:\, d(f_t(x),y)<r\right\}.\]

We will assume the existence of a Poincar\'e section $Y\subset\X$, and let $R_ Y(x)$ denote the first hitting time to $Y$, \ie $R_Y(x):=\inf\{t>0:f_t(x)\in Y\}$, with $\overline R:=\int R_Yd\nu<\infty$. Define the Poincar\'e map by $(Y,F,\mu)$ where $F=f_{R_Y}$ and $\mu$ is the induced measure on $Y$ given by $\mu=\frac1{\overline{R}}\nu|_Y$. Additionally, assume the following conditions are satisfied:
\begin{enumerate}[(H1)]
    \item \label{assumption H1}$\dim_M(\mu)$ exists and is finite for $(F,\mu)$,
    \item \label{assumption H2}$(Y,F,\mu)$ is Gibbs-Markov so \Cref{thm: finite minkowski dimension case} is applicable for $\mu$-almost every $y\in Y$.
    \item\label{assumption H3}  $\{f_t\}_t$ has bounded speed: there exists $K>0$ such that for all $t>0$, $d(f_s(x),f_{s+t}(x))<Kt$.
    \item\label{assumption H4} $\{f_t\}_t$ is topologically mixing and there exists $T_1>0$ such that
    \begin{equation}\label{eqn: union time}
        \bigcup_{0<t\le T_1}f_t(Y)=\X.
    \end{equation}
     \item\label{assumption H5} There exists 
    \[C_f:=\sup\left\{\text{diam}(f_t(I))/\text{diam}(I):I\text{ an interval contained in } Y, 0<t\le T_1\right\}\in(0,\infty)\]
\end{enumerate} 
\begin{remark}
The last condition is satisfied when \hyperref[assumption H3]{(H3)} holds and the flow is, for example, Lipschitz, \ie there exists $L>0$ such that for all $x,y\in \X$,
\[d_\X(f_t(x),f_t(y))\le L^td_\X(x,y).\]

\end{remark}

\begin{theorem}\label{thm: flow version}
Let $(f_t,\nu)$ be a probability preserving flow satisfying conditions \hyperref[assumption H1]{(H1)}-\hyperref[assumption H5]{(H5)}, then for $\nu$-almost every $x\in\Omega$,
    \begin{equation}\label{eqn: lower bound flow}\liminf_{r\to0}\frac{\log\tau_r(x)}{-\log r}\ge \underline{dim}_M(\nu)-1.\end{equation}
    Furthermore, if $\overline\dim_M(\nu)=\dim_M(\mu)+1$, 
    \begin{equation}\label{eqn: upper bound flow}\limsup_{r\to0}\frac{\log \tau_r(x)}{-\log r}\le\overline\dim_M(\mu)\hspace{2mm}\text{ $\nu$-a.e.}
    \end{equation}
\end{theorem}

\begin{proof}[Proof of \eqref{eqn: lower bound flow}]
This proof is analogous to those of \Cref{Prop lower bound 1} and \cite[Theorem 4.1]{RouTod24}. Let 

Fix some $y\in \Omega$ and $r>0$ and consider the random variable 
\[S_{T,r}(x):=\int_0^T\I_{B(y,r)}(f_t(x))dt,\]
and observe that by bounded speed property, for all $T>r/K$,
\[\left\{x:\exists 0\le t\le T \text{ s.t. }\,d(f_t(x),y))<r\right\}\subset\left\{S_{2T,2r}(x)>r/K\right\},\]
since if $d(f_s(x),y)<r$ for some $s$, then for all $t<r/K$, $d(f_{t+s}(x),y)<2r$. Also set \[T(x,y,r):=\inf\{t\ge0:f_t(x)\in B(y,r)\},\]and similarly for all $r>0$ and all $x,z$, $\tau_{r}(x)\ge T(x,y,r)$. 

Let $\eps>0$ be arbitrary and by definition of $\underline\alpha$ for all large $n\in\N$ there exists $y_n\in\Omega$ such that $\nu(B(y_n,2^{-n}))\le 2^{-n(\underline\alpha-\eps)}$. By Markov's inequality, for some $\mathcal T_n>0$ to be decided later,
\begin{align*}
    &\nu\left(x:\tau_{2^{-n}}(x)< \mathcal T_n\right)\le \nu\left(x:T(x,y_n,2^{-n})<  \mathcal T_n\right)=\nu\left(x:\exists\,0\le t< \mathcal T_n:\,f_t(x)\in B(y_n,2^{-n})\right)\\
    &\le \nu\left(x:S_{2\mathcal T_n,2^{-n+1}}(x)>r_n/K\right)\le  K2^n\int_0^{2\mathcal T_n}\int\I_{B(y_n,2^{-n+1})}(f_t(x))d\nu(x)dt\\
    &\le K2^{n+1}\mathcal T_n\nu(B(y_n,2^{-n+1}))\le 4K\mathcal T_n2^{-(n-1)(\underline{\alpha}-\eps-1)}.
\end{align*}
Choosing $\mathcal T_n=2^{(n-1)(\underline\alpha-\eps-1)}/n^2$, the last term above is summable along $n$ hence by Borel-Cantelli, for $\nu$-almost every $x$
\[\liminf_{r\to0}\frac{\log \tau_r(x)}{-\log r}\ge\liminf_{n\to\infty}\frac{\log\mathcal T_n}{n\log2}=\underline\alpha-1-\eps,\]
and since $\eps>0$ was arbitrary one can send it to 0, and by \Cref{remark subsequence trick} the proposition is proved.\qedhere
\end{proof}
Note that the proof of lower bound is independent of the existence or mixing properties of the Poincar\'e map $(Y,F,\mu)$. For upper bound, we first prove that the cover time of the Poincar\'e $F$ in $Y$ is comparable to the cover time of the flow.
\begin{lemma}\label{lemma: induced ct > original ct}
    Define
    \[\tau_r^F(x):=\min\{n\in\N_0:\forall y\in Y,\,\exists 0\le j\le n:d(y,F^jx)<r \}.\]
    There exists $\lambda=\frac1{C_f}$ for $C_f$ defined in \hyperref[assumption H5]{(H5)} such that $\tau_{r}(x)\le T_1+\sum_{j=0}^{\tau_{\lambda r}^F(x)}R_Y(F^jx).$
\end{lemma}
%
\begin{proof}
    This is adapted from the proof of \cite[Lemma 6.4]{JurTod24} and \cite[Theorem 2.1]{RouTod24}. 
   $F$ is by assumption Gibbs-Markov so one can find $\P(r)$, a natural partition of $Y$ using cylinder sets with respect to $F$ such that  for each $P\in\P(r)$: (a) $ \text{diam}(P)\le r/C_f$, and (b) for all $0<t\le T_1$, $f_t(P)$ is connected. Suppose $\tau^F_{r/C_f}(x)=k$, then the orbit $\{x,F(x),\dots,F^k(x)\}$ must have visited every element of $\P$, and by \eqref{eqn: union time} for each $y\in\Omega$ there is $P\in\P(r)$ and $0<s\le T_1$ such that $y\in f_s(P)$ and hence there exists $j\le k$ such that $d\left(f_s(F^j(x)),y\right)\le C_f|P|<r.$ Then set $\lambda=1/C_f$ the lemma is proved.
\end{proof}
\begin{proof}[Proof of \eqref{eqn: upper bound flow}]   
    Now assume $\overline\dim_M(\nu)=\dim_M(\mu)+1$.
    Let $\xi>0$ be arbitrary and define the sets \[U_{\xi,N}:=\left\{x\in Y:|R_n(x)-n\overline{R}|\le \xi n,\forall n\ge N\right\},\] where $R_n(x)=\sum_{j=0}^{n-1}R_Y(F^j(x))$. By ergodicity, $\lim_N\mu(U_{\xi,N})=1$ so for $N$ large, $\nu(U_{\xi,N})>0$ hence by invariance,    
    \begin{equation}\label{eqn: entry time}\lim_{N\to\infty}\nu\left(\bigcup_{t=0}^{\xi N}f_{-t}(U_{\xi,N})\right)=1.\end{equation}

    Let $\eps>0$ be arbitrary. By \eqref{eqn: entry time} one can pick $N^*$ such that for each $\nu$ typical $x\in \X$ there is some $t^*\le \xi N^*$ such that $f_{t^*}(x)\in Y$. By \Cref{thm: finite minkowski dimension case} applied to the Poincar\'e map and \Cref{lemma: induced ct > original ct}, for all sufficiently small $r>0$ we have the following two inequalities,
    \[\frac{\log \tau_{\lambda r}^F(f_{t^*}x)}{-\log\lambda r}\le \dim_M(\mu)+\eps,\hspace{2mm}\frac{\log \left(\tau_{r}(x)-T_1\right)}{-\log r}\le \frac{\log \left((\overline R+\xi)\tau_{\lambda r}^F(f_{t^*}x)\right)}{-\log r}. \]
    Then as $\lambda, \overline R$ are constants and $\eps$ is arbitrary, for $\nu$-almost every $x$,
    \[\limsup_{r\to0}\frac{\log \tau_r(x)}{-\log r}\le \dim_M(\mu)=\overline\dim_M(\nu)-1.\qedhere\]
\end{proof}

\subsection{Example: suspension semi-flows over topological Markov shifts}
In this section, we give an example of a flow for which $\dim_M(\nu)=\dim_M(\mu)+1$ is satisfied, so \Cref{thm: flow version} is applicable.

Let $\mathcal{A}$ be a {finite} alphabet and $M$ an $\mathcal{A}\times\mathcal{A}$ matrix with $\{0,1\}$ entries, we will consider two-sided topological Markov shift systems $(\Sigma,\sigma,\phi,\mu)$, where 
\[\Sigma:=\left\{x=(\dots,x_{-1},x_0,x_1,\dots)\in\mathcal{A}^{\mathbbm Z}:\text{ for all $j$, $x_j\in\mathcal{A}$ and $M_{x_j,x_{j+1}}=1$} \right\},\]
$\sigma$ the usual left shift, $\phi$ a H\"older potential and $\mu$ is the unique Gibbs measure with respect to $\phi$. We assume that $\dim_M(\mu)\in(0,\infty)$. The natural symbolic metric on $\Sigma$ is $d(x,y)=2^{-x\land y}$, where \[x\land y=\sup\{k\ge0:x_j=y_j,\,\forall |j|<k\}.\]
An $n$-cylinder in this setting is given by
$\left[x_{-(n-1)},\dots,x_0,\dots,x_{n-1}\right]:=\left\{y\in\Sigma, \,y_j=x_j,\forall |j|<n\right\}$, and it is a well-known fact that balls in $\Sigma$ are precisely the cylinder sets. The left-shift map $\sigma$ is bi-Lipschitz with Lipschitz constant $L=2$. For more detailed description of the shift space, see \cite[\S1]{Bow75}.
 
Let $\varphi\in L^1(\mu)$ be a positive Lipschitz function, define the space 
\[Y_\varphi:=\left\{(x,s)\in \Sigma\times \R_{\ge0}:0\le s\le \varphi(x)\right\}{/\sim}\]
where $(x,\varphi(x))\sim (\sigma(x),0)$ for all $x\in I$. The suspension flow $\Psi$ over $\sigma$ is the function acts on $Y_\varphi$ by
\[\Psi_t(x,s)=(\sigma^k(x),v),\]
where $k,v\ge0$ are determined by $s+t=v+\sum_{j=0}^{k-1}\varphi(\sigma^j(x)).$ The invariant measure $\nu$ for the flow $\Psi$ on $Y_\varphi$ satisfies the following: for every $g:Y_\varphi\to\R$ continuous,
\begin{equation}\label{eqn: decomposition of flow measure}
\int gd\nu=\frac1{\int_\Sigma\varphi d\mu}\int_\Sigma\int_0^{\varphi(x)}g(x,s)dsd\mu(x).\end{equation}

The metric on $Y_\varphi$ is the Bowen-Walters distance $d_Y$ (see for example \cite{BowWal72}).
Define another metric $d_\pi$ on $Y_\varphi$: for all $(x_i,t_i)_{i=1,2}\in Y_\varphi$, 
\[d_\pi((x_1,t_1),(x_2,t_2)):=\min\left\{\begin{aligned}
    &d(x,y)+|s-t|,\\
    &d(\sigma x,y)+\varphi(x)-s+t,\\
    &d(x,\sigma y)+\varphi(y)-t+s
    \end{aligned}\right\},\]
    and the following proposition says $d_\pi$ is comparable to the Bowen-Walters distance.
\begin{proposition}\cite[Proposition 17]{BarSau00}\label{prop: alternate metric}
There exists $c=c_\pi$ such that 
    \[c^{-1}d_\pi((x_1,t_1),(x_2,t_2))\le d_Y((x_1,t_1),(x_2,t_2))\le c\, d_\pi((x_1,t_1),(x_2,t_2)).\]
\end{proposition}
Then the Minkowski dimension of the flow-invariant measure $\nu$ is given by 
\begin{proposition}$\dim_M(\nu)=\dim_M(\mu)+1$.
\end{proposition}
\begin{proof}
The proof is based on the proof of \cite[Theorem 4.3]{RouTod24} for correlation dimensions.

By \Cref{prop: alternate metric} for all $r>0$, \[\left(B(x,r/2c)\times(s-r/2c,s+r/2c)\right)\cap Y\subset B_Y((x,s),r)\] where $B_Y$ denotes the ball with respect to the metric $d_Y$, then for all $(x,s)\in Y_\varphi$, put $\overline\varphi=\int_\Sigma\varphi d\mu$, then
\[\nu(B_Y((x,s),r))\ge \nu(B(x,r/2c)\times\left(s-\frac{r}{2c},s+\frac{r}{2c}\right),\]
\[\frac{\log\nu(B_Y((x,s),r))}{\log r}\le \frac{\log\left(\frac{r}{c\overline\varphi}\mu\left(B(x,\frac r{2c})\right)\right)}{\log r}.\]
Hence $\overline{\dim}_M(\nu)=\limsup_{r\to0}\frac{\log \min_{(x,s)\in supp(\nu)}\nu(B_Y((x,s),r)}{\log r}\le \dim_M(\mu)+1$.

For lower bound, define
\[B_1:=B(x,c r)\times(s-c r,s+c r),\hspace{2mm}B_2:=B(\sigma x,c r)\times[0,cr),\]
\[B_3:=\left\{(y,t):y\in B(\sigma^{-1}x,2cr),\text{ and }\varphi(y)-cr\le t\le \varphi(y)\right\}.\]
Then as in the proof of \cite[Theorem 4.3]{RouTod24}, $B_Y((x,s),r)\subset\left(B_1\cup B_2\cup B_3\right)\cap Y_\varphi$.


For all $r>0$ and $(x,s)\in Y_\varphi$ by \eqref{eqn: decomposition of flow measure}, and as $\mu$ is $\sigma,\sigma^{-1}$ invariant,
\begin{align*}
&\nu(B_1\cap Y_\varphi)={2cr}\mu(B(x,cr))/{\overline\varphi},\hspace{2mm} \nu(B_2,Y_\varphi)\le cr\mu(B(x,cr))/\overline\varphi\\
&\nu(B_3\cap Y_\varphi)\le cr\mu(\sigma^{-1}B(x,2cr))/\overline\varphi=cr\mu(B(x,2cr))/\overline\varphi. 
\end{align*}
Therefore 
\[\nu(B_Y((x,s),r)\le\frac1{{\overline\varphi}}\left(3r\mu(B(x,cr))+cr\mu(B(x,2cr))\right),\]
which is enough to conclude that $\underline{dim}_M(\nu)\ge\dim_M(\mu)+1$.
\end{proof}

\section*{Acknowledgement}
I acknowledge the grant from \emph{Chinese Scholarship Council}. I am also thankful for various comments and help from my supervisor M. Todd, as well as other comments on \Cref{section: flows} from J. Rousseau.


\begin{thebibliography}{XXX}
\bibitem[BJK]{BarJurKlo}B. Bárány, N. Jurga and I. Kolossváry.\emph{On the Convergence Rate of the Chaos Game}, International Mathematics Research Notices, \textbf{2023} 5, (2023) 4456–4500.

\bibitem[BS]{BarSau00} L. Barreira, B. Saussol. \emph{Multifractal analysis of hyperbolic flows}, Comm. Math. Phys \textbf{214} (2000), 339-371.

\bibitem[BW]{BowWal72} R.Bowen, P. Walters. \emph{Expansive one-parameter flows}, J. Differential Equations, \textbf{12} (1972), 180-193.

\bibitem[Bow]{Bow75} R. Bowen, \emph{Equilibrium States and The Ergodic Theory of Anosov Diffeomorphisms}, Lect. Notes in Math. \textbf{470}, Springer, 1975.

\bibitem[FFK]{FalFraKae} K. Falconer, J. Fraser and A. K\"aem\"aki. \emph{Minkowski dimension for measures}, 	Proc. Amer. Math. Soc.,\textbf{ 151} (2023), 779-794.


\bibitem[GRS]{GouRouSta}S. Gouëzel, J. Rousseau and M. Stadlbauer. \emph{Minimal distance between random orbits}, Probab. Theory Relat. Fields (2024).

\bibitem[JT]{JurTod24}N. Jurga, M. Todd. \emph{Cover times in dynamical systems}, In Press: Isarel Journal of Mathematics, (2024).

\bibitem[K]{Khi64} A. Khintchine. \emph{Continued Fractions}, Univ. Chicago Press, Chicago, (1964).

\bibitem[KS]{KimSeo03}D. Kim and B. Seo. \emph{The waiting time for irrational rotations}, Nonlinearity, \textbf{16} 1861 (2003).

\bibitem[M]{Mat88}P. Matthews. \emph{Covering problems for Brownian motion on spheres}, Ann. Probab. \textbf{(1)} 16 (1988) 189–199.

\bibitem[RT]{RouTod24} J. Rousseau, M. Todd. \emph{Orbits closeness for slowly mixing dynamical systems}. Ergodic Theory and Dynam. Systems, \textbf{44} 4 (2024), 1192–1208.
\end{thebibliography}
\end{document}